\documentclass[10pt]{article}
\usepackage{amsmath, amsthm, amssymb, amsfonts, mathrsfs, mathtools}
\usepackage{graphicx}
\usepackage{epstopdf}
\usepackage{makeidx}
\usepackage[left=2cm,right=2cm,top=2cm,bottom=2cm]{geometry}
\usepackage{caption}
\usepackage{subcaption}
\usepackage[section]{placeins}
\usepackage{enumitem}
\usepackage{tikz}
\usepackage{stmaryrd}

\newcommand{\Mod}[1]{\ (\mathrm{mod}\ #1)}

\usetikzlibrary{decorations.pathreplacing}
\tikzstyle{vertex}=[auto=left,circle,draw=black,fill=white, inner sep=1.5]

\usepackage{hyperref}
\hypersetup{colorlinks=true, citecolor={blue}, linkcolor=blue, filecolor=magenta, urlcolor=cyan}

\newtheorem{theorem}{Theorem}[section]
\newtheorem{cor}[theorem]{Corollary}

\newtheorem{lema}{Lemma} [section]

\title{Integral mixed cayley graph over abelian group }

\author{ Monu Kadyan\\
Department of Mathematics\\
Indian Institute of Technology Guwahati\\
Guwahati, India - 781039\\
Email: monu.kadyan@iitg.ac.in\\
\\
Bikash Bhattacharjya\\
Department of Mathematics\\
Indian Institute of Technology Guwahati\\
Guwahati, India - 781039\\
Email: b.bikash@iitg.ac.in 
}

\date{}

\begin{document}
\maketitle

\vspace{-0.3in}

	\begin{center}{\textbf{Abstract}}\end{center}

A mixed graph is said to be \textit{integral} if all the eigenvalues of its Hermitian adjacency matrix are integer. Let $\Gamma$ be an abelian group. The \textit{mixed Cayley graph} $Cay(\Gamma,S)$ is a mixed graph on the vertex set $\Gamma$ and edge set $\{ (a,b): b-a\in S \}$, where $0\not\in S$. We characterize integral mixed Cayley graph $Cay(\Gamma,S)$ over abelian group in terms of its symbol set $S$.

\vspace*{0.3cm}

\section{Introduction}
We only consider graphs without loops and  multi-edges. A  \textit{(simple) graph} $G$ is denoted by $G=(V(G),E(G))$, where $V(G)$ and $E(G)$ are the vertex set and edge set of $G$, respectively. Here $E(G)\subset V(G) \times V(G)\setminus \{(u,u)|u\in V(G)\}$ such that $(u,v)\in E(G)$ if and only if $(v,u)\in E(G)$. 
A graph $G$ is said to be \textit{oriented} if $(u,v)\in E(G)$ implies that $(v,u)\not\in E(G)$. A graph $G$ is said to be \textit{mixed} if $(u,v)\in E(G)$ does not always imply that $(v,u)\in E(G)$, see \cite{2015mixed} for details. In a mixed graph $G$, we call an edge with end vertices $u$ and $v$ to be \textit{undirected} (resp. \textit{directed}) if both $(u,v)$ and $(v,u)$ belong to $E(G)$ (resp. only one of $(u,v)$ and $(v,u)$ belongs to $E(G)$). An undirected edge $(u,v)$ is denoted by $u\leftrightarrow v$, and a directed edge $(u,v)$ is denoted by $u\rightarrow v$. A mixed graph can have both directed and undirected edges. Note that, if all edges of a mixed graph G are directed (resp. undirected) then G is an oriented graph (resp. a simple graph). For a mixed graph $G$, the underlying graph $G_{U}$ of $G$ is the simple undirected graph in which all edges of $G$ are considered undirected. By the terms of order, size, number of components, degree of a vertex, distance between two vertices etc., we mean that they are the same as in their underlying graphs.

In 2015, Liu and Li \cite{2015mixed} introduced Hermitian adjacency matrix of a mixed graph. For a mixed graph with $n$ vertices, its \textit{Hermitian adjacency matrix} is denoted by $H(G)=(h_{uv})_{n\times n}$, where $h_{uv}$ is given by
$$h_{uv} = \left\{ \begin{array}{rl}
	1 &\mbox{ if }
	(u,v)\in E \textnormal{ and } (v,u)\in E, \\ i & \mbox{ if } (u,v)\in E \textnormal{ and } (v,u)\not\in E ,\\
	-i & \mbox{ if } (u,v)\not\in E \textnormal{ and } (v,u)\in E,\\
	0 &\textnormal{ otherwise.}
\end{array}\right.$$ 
Here $i=\sqrt{-1}$ is the imaginary number unit. Hermitian adjacency matrix of a mixed graph incorporates both adjacency matrix of simple graph and skew adjacency matrix of an oriented graph. The Hermitian spectrum of $G$, denoted by $Sp_H(G)$, is the multi set of the eigenvalues of $H(G)$. It is easy to see that $H(G)$ is a Hermitian matrix and so $Sp_H(G)\subseteq \mathbb{R}$. 

A mixed graph is said to be \textit{integral} if all the eigenvalues of its Hermitian adjacency matrix are integers. Integral graphs were first defined by Harary and Schwenk in 1974 \cite{harary1974graphs} and proposed a classification of integral graphs. See \cite{balinska2002survey} for a survey on integral graphs. %Watanabe and Schwenk \cite{watanabe1979note,watanabe1979integral} proved several interesting results on integral trees in 1979. Csikvari \cite{csikvari2010integral} constructed integral trees with arbitrary large diameters in 2010. Further research on integral trees can be found in \cite{brouwer2008integral,brouwer2008small,wang2000some, wang2002integral}. In $2009$, Ahmadi et al. \cite{ahmadi2009graphs} proved that only a fraction of $2^{-\Omega (n)}$ of the graphs on $n$ vertices have an integral spectrum. Bussemaker et al. \cite{bussemaker1976there} proved that there are exactly $13$ connected cubic integral graphs. Stevanovi{\'c} \cite{stevanovic20034} studied the $4$-regular integral graphs avoiding $\pm3$ in the spectrum, and Lepovi{\'c} et al. \cite{lepovic2005there} proved that there are $93$ non-regular, bipartite integral graphs with maximum degree four. In 2017, Guo et. al. \cite{2017mixed}  found all possible mixed graphs on $n$ vertices with spectrum $(-n+1,-1,-1,...,-1)$.

Let $\Gamma$ be a group, $S \subseteq \Gamma$ and $S$ does not contain the identity element of $\Gamma$. The set $S$ is said to be \textit{symmetric} (resp. \textit{skew-symmetric}) if $S$ is closed under inverse (resp. $a^{-1} \not\in S$ for all $a\in S$). Define $\overline{S}= \{u\in S: u^{-1}\not\in S \}$. Clearly $S\setminus \overline{S}$ is symmetric and $\overline{S}$ is skew-symmetric. The \textit{mixed Cayley graph} $G=Cay(\Gamma,S)$ is a mixed graph, where $V(G)=\Gamma$ and $E(G)=\{ (a,b):a,b\in \Gamma, ba^{-1}\in S \}$.  Since we have not assumed that $S$ is symmetric, so a mixed Cayley graph can have directed edges. If $S$ is symmetric, then $G$ is a (simple) \textit{Cayley graph}. If $S$ is skew-symmetric then $G$ is an \textit{oriented Cayley graph}.

In 1982, Bridge and Mena \cite{bridges1982rational} introduced a characterization of integral Cayley graphs over abelian groups. Later on, the exact characterization was rediscovered by Wasin So \cite{2006integral} for cyclic groups in 2005.
In 2009, Abdollahi and Vatandoost \cite{abdollahi2009cayley} proved that there are exactly seven connected cubic integral Cayley graphs. In the same year, Klotz and Sander \cite{klotz2010integral} proved that if the Cayley graph $Cay(\Gamma,S)$ over abelian group $\Gamma$ is integral, then $S$ belongs to the Boolean algebra $\mathbb{B}(\Gamma)$ generated by the subgroups of $\Gamma$, and its converse proved by Alperin and Peterson \cite{alperin2012integral}. In 2014, Cheng et al. \cite{ku2015cayley} proved that normal Cayley graphs (its generating set $S$ is closed under conjugation) of symmetric groups are integral. Alperin \cite{alperin2014rational} gave a characterization of integral Cayley graphs over finite groups.  In 2017, Lu et al. \cite{lu2018integral} gave necessary and sufficient conditions for the integrality of Cayley graphs over dihedral groups $D_n$. In particular, they completely determined all integral Cayley graphs of the dihedral group $D_p$ for a prime $p$. In 2019, Cheng et al.\cite{cheng2019integral} obtained several simple sufficient conditions for the integrality of Cayley graphs over dicyclic group $T_{4n}= \langle a,b| a^{2n}=1, a^n=b^2,b^{-1}ab=a^{-1} \rangle $. In particular, they also completely determined all integral Cayley graphs over the dicyclic group $T_{4p}$ for a prime $p$. In \cite{kadyan2021integral}, the authors have characterized integral mixed circulant graphs in terms of their symbol set. In this paper, we give a characterization of integral mixed Cayley graph over abelian groups in terms if its symbol set. In what folows, $\Gamma$ is always taken to be a finite abelian group.

This paper is organized as follows. In second section, we express the eigenvalues of a mixed Cayley graph as a sum of eigenvalues of a simple Cayley graph and an oriented Cayley graph. In third section, we obtain a sufficient condition on the symbol set $S$ for integrality of the mixed Cayley graph $Cay(\Gamma,S)$ over abelian group $\Gamma$. In fourth section, we prove the necessity of the sufficient condition obtained in Section 3.

 %%%%%%%%%%%%%%%%%%%%%%%%%%%%%%%%%%%%%%%%%%%%%%%%%%%%%%%%%%%%%%%%%%%%%%%%%%%%%%%%%%%%%%%%%%%

\section{Mixed Cayley graph and group characters} 

A \textit{representation} of a finite group $\Gamma$ is a homomorphism $\rho : \Gamma \to GL(V)$, where $GL(V)$ is  the group of automorphisms of a finite dimensional vector space $V$ the complex field $\mathbb{C}$. The dimension of $V$ is called the \textit{degree} of $\rho$. Two representations $\rho_1$ and $\rho_2$ of $\Gamma$ on $V_1$ and $V_2$, respectively, are \textit{equivalent} if there is an isomorphism $T:V_1 \to V_2$ such that $T\rho_1(g)=\rho_2(g)T$ for all $g\in \Gamma$.

Let $\rho : \Gamma \to GL(V)$ be a representation. The \textit{character} $\chi_{\rho}: \Gamma \to \mathbb{C}$ of $\rho$ is defined by setting $\chi_{\rho}(g)=Tr(\rho(g))$ for $g\in \Gamma$, where $Tr(\rho(g))$ is the trace of the representation matrix of $\rho(g)$. By degree of $\chi_{\rho}$ we mean the degree of $\rho$ which is simply $\chi_{\rho}(1)$. If $W$ is a $\rho(g)$-invariant subspace of $V$ for each $g\in \Gamma$, then we say $W$ a $\rho(\Gamma)$-invariant subspace of $V$. If the only $\rho(\Gamma)$-invariant subspaces of $V$ are $\{ 0\}$ and $V$,  we say $\rho$ an \textit{irreducible representation} of $\Gamma$, and the corresponding character $\chi_{\rho}$ an \textit{irreducible character} of $\Gamma$. 

For a group $\Gamma$, we denote by $IRR(\Gamma)$ and $Irr(\Gamma)$ the complete set of non-equivalent irreducible representations of $\Gamma$  and the complete set of non-equivalent irreducible characters of $\Gamma$, respectively.

 Let $\Gamma$ be a finite abelian group under addition with $n$ elements, and $S$ be a subset of $\Gamma$ with $0\not\in S$, where $0$ is the additive identity of $\Gamma$.  Then $\Gamma$ is isomorphic to the direct product of cyclic groups of prime power order, $i.e.$ 
 $$\Gamma\cong \mathbb{Z}_{n_1} \otimes \cdots \otimes \mathbb{Z}_{n_k},$$
 where $n=n_1 \cdots n_k$, and $n_j$ is a power of a prime number for each $j=1,...,k  $. We consider an abelian group $\Gamma$ as $\mathbb{Z}_{n_1} \otimes \cdots \otimes \mathbb{Z}_{n_k}$ of order $n=n_1...n_k$. The \textit{exponent} of $\Gamma$ is defined to be the least common multiple of $n_1,n_2,...,n_k$, denoted by $exp(\Gamma)$. We consider the elements $x\in \Gamma $ as elements of the cartesian product $\mathbb{Z}_{n_1} \otimes \cdots \otimes \mathbb{Z}_{n_k}$, $i.e.$ 
 $$x=(x_1,x_2,...,x_k),  \mbox{ where } x_j \in \mathbb{Z}_{n_j} \mbox{ for all } 1\leq j \leq k. $$
Addition in $\Gamma$ is done coordinate-wise modulo $n_j$. For a positive integer $k$ and $a\in \Gamma$ we denote by $ka$ or $a^k$ the $k$-fold sum of $a$ to itself, $(-k)a=k(-a)$, $0a=0$, and inverse of $a$ by $-a$. 
%The \textit{mixed Cayley graph} on $\Gamma$ with respect to S, denoted by $Cay(\Gamma,S)$, is the mixed graph on the set of vertices $\Gamma$, and edge set $\{ (v,v+s)| v\in \Gamma, s\in S \}$. If, in addition, $S$ is inverse-closed $(i.e., S=S^{-1}=\{ -s| s\in S\})$, then $Cay(\Gamma,S)$ can be regarded as an undirected graph, called \textit{Cayley graph} or \textit{simple Cayley graph}. If $S$ does not contain $-a$ for every $a\in S$, then $Cay(\Gamma,S)$ can be regarded as a directed graph, called \textit{oriented Cayley graph}.

\begin{lema}\label{lemma1}\cite{steinberg2009representation}
Let $\mathbb{Z}_n=\{ 0,1,...,n-1\}$ be a cyclic group of order $n$. Then $IRR(\mathbb{Z}_n)=\{ \phi_k: 0\leq k \leq n-1\}$, where $\phi_k(j)=w_n^{jk}$ for all $0\leq j,k \leq n-1$, and $w_n=\exp(\frac{2\pi i}{n})$.
\end{lema}

\begin{lema}\label{lemma2}\cite{steinberg2009representation}
Let $\Gamma_1$,$\Gamma_2$ be abelian groups of order $m,n$, respectively. Let $IRR(\Gamma_1)=\{ \phi_1,...,\phi_m\}$, and $IRR(\Gamma_2)=\{ \rho_1,...,\rho_n\}$. Then $IRR(\Gamma_1 \times \Gamma_2)=\{ \psi_{kl} : 1\leq k \leq m, 1\leq l \leq n \}$, where $\psi_{kl}: \Gamma_1 \times \Gamma_2 \to \mathbb{C}^* \mbox{ and } \psi_{kl}(g_1,g_2)=\phi_k(g_1)\rho_l(g_2)$ for all $g_1\in \Gamma_1, g_2\in \Gamma_2$.
\end{lema}

Consider $\Gamma = \mathbb{Z}_{n_1}\times \mathbb{Z}_{n_2}\times ...\times  \mathbb{Z}_{n_k}$. By Lemma \ref{lemma1} and Lemma \ref{lemma2}, $IRR(\Gamma)=\{ \psi_{\alpha}: \alpha \in \Gamma\}$, where $$\psi_{\alpha}(x)=\prod_{j=1}^{k}w_{n_j}^{\alpha_j x_j} \textnormal{ for all $\alpha=( \alpha_1,...,\alpha_k),x=(x_1,...,x_k) \in \Gamma$},$$ and $w_{n_j}=\exp(\frac{2\pi i}{n_j})$. Since $\Gamma$ is an abelian group, every irreducible representation of $\Gamma$ is 1-dimensional and thus it can be identified with its characters. Hence $IRR(\Gamma)=Irr(\Gamma)$. For $x\in \Gamma$, let $ord(x)$ denote the order of $x$. The following statements can be easily proved.

\begin{lema}\label{Basic}
Let $\Gamma$ be an abelian group of order $n$, and $Irr(\Gamma)=\{\psi_{\alpha} : \alpha \in \Gamma \}$ be the set of all $n$ characters of $\Gamma$. Then the following statements are true.
\begin{enumerate}[label=(\roman*)]
\item $\psi_{\alpha}(x)=\psi_x({\alpha})$ for all $x,\alpha \in \Gamma$.
\item $(\psi_{\alpha}(x))^{ord(x)}=(\psi_{\alpha}(x))^{ord(\alpha)}=1$ for all $x,\alpha \in \Gamma$.
\item $\psi_{\alpha}(x)^l=1$ for all $x,\alpha \in \Gamma$, where $l=exp(\Gamma)$.
\end{enumerate} 
\end{lema}

\begin{lema}\cite{klotz2010integral} \label{OrthogonalBasis}
Let $\Gamma=\{ v_1,...,v_n\}$ be an abelian group, and $X_{\alpha}= \begin{bmatrix} \psi_{\alpha} (v_1), &  \psi_{\alpha} (v_2), & \cdots & ,\psi_{\alpha} (v_n)\end{bmatrix}^t$ for all $\alpha \in \Gamma$. Then the set $\{ X_{\alpha} : \alpha \in \Gamma \}$ is an orthogonal basis of $\mathbb{C}^n$.
\end{lema}

\begin{lema}\label{IntegralChar}\label{imcgoa3}
Let $\Gamma$ be an abelian group. Then the spectrum of the mixed Cayley graph $Cay(\Gamma, S)$ is \linebreak[4] $\{ \gamma_\alpha : \alpha \in \Gamma \}$, where  $\gamma_{\alpha}=\lambda_{\alpha}+\mu_{\alpha}$, and $$\lambda_{\alpha}=\sum_{s\in S\setminus \overline{S}} \psi_{\alpha}(s),\hspace{1cm} \mu_{\alpha}=i\sum_{s\in\overline{S}}\bigg( \psi_{\alpha}(s)-\psi_{\alpha}(-s)\bigg).$$
\end{lema}
\begin{proof} For $\alpha \in \Gamma$, the $k$-th entry of $H(Cay(\Gamma,S))X_{\alpha}$ is given by 
\begin{equation*}
\begin{split}
\sum_{v_j \in \Gamma} h_{v_kv_j} \psi_{\alpha}(v_j) &= \sum_{1\leq j \leq n,v_j-v_k \in S \setminus \overline{S}} \psi_{\alpha}(v_j) +  \sum_{1\leq j \leq n,v_j-v_k \in \overline{S}} i\psi_{\alpha}(v_j) +  \sum_{1\leq j \leq n,v_k-v_j \in \overline{S}}(-i) \psi_{\alpha}(v_j) \\
&=  \sum_{s\in S \setminus \overline{S}} \psi_{\alpha}(s+v_k) + i \sum_{s \in \overline{S}} \psi_{\alpha}(s+v_k) -i  \sum_{s \in \overline{S}} \psi_{\alpha}(-s+v_k) \\
&= \bigg(  \sum_{s\in S\setminus \overline{S}} \psi_{\alpha}(s)+i\sum_{s\in\overline{S}}( \psi_{\alpha}(s)-\psi_{\alpha}(-s)) \bigg) \psi_{\alpha}(v_k)\\
&=(\lambda_{\alpha} + \mu_{\alpha}) \psi_{\alpha}(v_k)\\
&= \gamma_{\alpha}  \psi_{\alpha}(v_k).
\end{split}
\end{equation*} 
Thus $H(Cay(\Gamma,S))X_{\alpha}= \gamma_{\alpha} X_{\alpha}$, and so  $\gamma_{\alpha}$ is an eigenvalue of $H(Cay(\Gamma,S))$ with eigenvector $X_{\alpha}$. Using Lemma ~\ref{OrthogonalBasis}, the result follows. 
\end{proof}

Next two corollaries are special cases of Lemma~\ref{imcgoa3}.

\begin{cor}\cite{klotz2010integral}\label{klotz2010integral}
Let $\Gamma$ be an abelian group. Then the spectrum of the Cayley graph $Cay(\Gamma, S)$ is $\{ \lambda_\alpha : \alpha \in \Gamma \}$, where $\lambda_\alpha=\lambda_{-\alpha}$ and $$\lambda_{\alpha}=\sum_{s\in S} \psi_{\alpha}(s) \mbox{ for all } \alpha \in \Gamma.$$
\end{cor}
\begin{proof}
Note that $\overline{S}=\emptyset$ and use Lemma \ref{IntegralChar}. Also
\begin{equation*}
\begin{split}
\lambda_{\alpha}&=\sum_{s\in S} \psi_{\alpha}(s)=\sum_{s\in S} \psi_{-\alpha}(-s)=\sum_{s\in S} \psi_{-\alpha}(s)  =\lambda_{-\alpha}.
\end{split}
\end{equation*} 
\end{proof}

\begin{cor}\label{OriEig}
Let $\Gamma$ be an abelian group. Then the spectrum of the oriented Cayley graph $Cay(\Gamma, S)$ is $\{ \mu_\alpha : \alpha \in \Gamma \}$, where $\mu_\alpha=-\mu_{-\alpha}$ and $$\mu_{\alpha}=i\sum_{s\in S}\bigg( \psi_{\alpha}(s)-\psi_{\alpha}(-s)\bigg) \mbox{ for all } \alpha \in \Gamma.$$
\end{cor}
 \begin{proof}
Note that $S \setminus \overline{S}=\emptyset$ and use Lemma \ref{IntegralChar}. Also
 \begin{equation*}
\begin{split}
\mu_{\alpha}&=i\sum_{s\in {S}}\bigg( \psi_{\alpha}(s)-\psi_{\alpha}(-s)\bigg)=i\sum_{s\in {S}}\bigg( \psi_{-\alpha}(-s)-\psi_{-\alpha}(s)\bigg)=-\mu_{-\alpha}.
\end{split}
\end{equation*}
 \end{proof}

\begin{theorem}\label{smallchara}
Let $\Gamma$ be an abelian group. The mixed Cayley graph $Cay(\Gamma, S)$ is integral if and only if both Cayley graph $Cay(\Gamma, S\setminus \overline{S})$ and oriented Cayley graph $Cay(\Gamma,  \overline{S})$ are integral.
\end{theorem}
\begin{proof}
Assume that the mixed Cayley graph $Cay(\Gamma, S)$ is integral. Let $\gamma_{\alpha}$ be an eigenvalue of mixed Cayley graph $Cay(\Gamma, S)$. By Lemma~\ref{IntegralChar}, Corollary~\ref{klotz2010integral} and Corollary~\ref{OriEig}, we have $\gamma_{\alpha}=\lambda_{\alpha}+\mu_{\alpha}$ and $\gamma_{-\alpha}=\lambda_{\alpha}-\mu_{\alpha}$ for all $\alpha \in \Gamma$, where $\lambda_{\alpha}$ is an eigenvalue of the Cayley graph $Cay(\Gamma, S\setminus \overline{S})$ and $\mu_{\alpha}$ is an eigenvalue of the oriented Cayley graph $Cay(\Gamma,  \overline{S})$. Thus $\lambda_{\alpha}=\frac{\gamma_{\alpha}+\gamma_{-\alpha}}{2} \in \mathbb{Q}$ and $\mu_{\alpha}=\frac{\gamma_{\alpha}-\gamma_{-\alpha}}{2} \in \mathbb{Q}$. As $\lambda_{\alpha}$ and $\mu_{\alpha}$ are rational algebraic integers, so $\lambda_{\alpha}, \mu_{\alpha}\in \mathbb{Q}$ implies that $\lambda_{\alpha}$ and $\mu_{\alpha}$ are integers. Thus the Cayley graph $Cay(\Gamma, S\setminus \overline{S})$ and the oriented Cayley graph $Cay(\Gamma,  \overline{S})$ are integral. 

Conversely, assume that both Cayley graph $Cay(\Gamma, S\setminus \overline{S})$ and oriented Cayley graph $Cay(\Gamma,  \overline{S})$ are integral. Then Lemma~\ref{IntegralChar} implies that $Cay(\Gamma, S)$ is integral.
\end{proof}

Let $n\geq 2$ be a fixed positive integer.  Define $G_n(d)=\{k: 1\leq k\leq n-1, \gcd(k,n)=d \}$. It is clear that $G_n(d)=dG_{\frac{n}{d}}(1)$.

Alperin and Peterson \cite{alperin2012integral} considered a Boolean algebra generated by a class of subgroups of a group in order to determine the integrality of Cayley graphs over abelian groups. Suppose $\Gamma$ is a finite group, and $\mathcal{F}_{\Gamma}$ is the family of all subgroups of $\Gamma$. The Boolean algebra $\mathbb{B}(\Gamma)$ generated by $\mathcal{F}_{\Gamma}$ is the set whose elements are obtained by arbitrary finite intersections, unions, and complements of the elements in the family $\mathcal{F}_{\Gamma}$. The minimal non-empty elements of this algebra are called \textit{atoms}. Thus each element of $\mathbb{B}(\Gamma)$ is the union of some atoms.
Consider the equivalence relation $\sim$ on $\Gamma$ such that $x\sim y$ if and only if $y=x^k$ for some $k\in G_m(1)$, where $m=ord(x)$.

\begin{lema}\cite{alperin2012integral}  The equivalence classes of $\sim$ are the atoms of $\mathbb{B}(\Gamma)$.
\end{lema}

For $x\in \Gamma$, let $[x]$ denote the equivalence class of $x$ with respect to the relation $\sim$. Also, let $\langle x \rangle$ denote the cyclic group generated by $x$.

\begin{lema}\label{atomsboolean} \cite{alperin2012integral}  The atoms of the Boolean algebra $\mathbb{B}(\Gamma)$ are the sets $[x]=\{ y: \langle y \rangle = \langle x \rangle \}$.
\end{lema}

By Lemma \ref{atomsboolean}, each element of $\mathbb{B}(\Gamma)$ is a union of some sets of the form $[x]=\{ y: \langle y \rangle = \langle x \rangle \}$. Thus, for all $S\in \mathbb{B}(\Gamma)$, we have $S=[x_1]\cup...\cup [x_k]$ for some $x_1,...,x_k\in \Gamma$.

The next result provides a complete characterization of integral Cayley graphs over an abelian group $\Gamma$ in terms of the atoms of $\mathbb{B}(\Gamma)$. 

\begin{theorem}\label{Cayint} (\cite{alperin2012integral}, \cite{bridges1982rational})
Let $\Gamma$ be an abelian group. The Cayley graph $Cay(\Gamma, S)$ is integral if and only if $S\in \mathbb{B}(\Gamma)$.
\end{theorem}

%%%%%%%%%%%%%%%%%%%%%%%%%%%%%%%%%%%%%%%%%%%%%%%%%%%%%%%%%%%%%%%%%%%%%%%%%%%%%%%%%%%%%%%%%%%%%
\section{A Sufficient condition for integrality of mixed Cayley graph over abelian group}

Unless otherwise stated, we consider $\Gamma$ to be an abelian group of order $n$. Due to Theorem \ref{smallchara}, to find characterization of the integral mixed Cayley graph $Cay(\Gamma, S)$, it is enough to find characterization of the integral  Cayley graph $Cay(\Gamma, S\setminus \overline{S})$ and the integral oriented Cayley graph $Cay(\Gamma,  \overline{S})$. The integral Cayley graph $Cay(\Gamma, S\setminus \overline{S})$ is characterized by Theorem \ref{Cayint}. So our attempt is to characterize the integral oriented Cayley graph $Cay(\Gamma,  \overline{S})$. 

For all $x\in \Gamma(4)$ and $r\in \{0,1,2,3\}$, define $$M_r(x):=\{x^k: 1\leq k \leq ord(x) , k \equiv r \Mod 4  \}.$$
 For all $a\in \Gamma$ and $S\subseteq \Gamma$, define $a+S:= \{ a+s: s\in S\}$ and $-S:=\{ -s: s\in S \}$. Note that $-s$ denotes the inverse of $s$, that is $-s=s^{m-1}$, where $m=ord(s)$.

\begin{lema} Let $\Gamma$ be an abelian group and $x\in \Gamma(4)$. Then the following statement\textbf{}s are true.
\begin{enumerate}[label=(\roman*)]
\item $\bigcup\limits_{r=0}^{3}M_r(x)= \langle x \rangle$.
\item Both $M_1(x)$ and $M_3(x)$ are skew symmetric subsets of $\Gamma$.
\item $-M_1(x)=M_3(x)$ and $-M_3(x)=M_1(x)$.
\item $a+M_1(x)=M_3(x)$ and $a+M_3(x)=M_1(x)$ for all $a\in M_2(x)$.
\item $a+M_1(x)=M_1(x)$ and $a+M_3(x)=M_3(x)$ for all $a\in M_0(x)$.
\end{enumerate} 
\end{lema}
\begin{proof}
\begin{enumerate}[label=(\roman*)]
	\item It follows from the definitions of $M_r(x)$ and $\langle x \rangle$.
	\item If $x^k\in M_1(x)$ then $-x^k=x^{n-k} \not\in M_1(x)$, as $k\equiv 1 \pmod 4$ gives $n-k\equiv 3 \pmod 4$. Thus  $M_1(x)$ is a skew symmetric subset of $\Gamma$. Similarly,  $M_3(x)$ is also a skew symmetric subset of $\Gamma$.
	\item As $k\equiv 1 \pmod 4$ if and only if $n-k\equiv 3 \pmod 4$, we get $-x^k=x^{n-k}$. Therefore $-M_1(x)=M_3(x)$ and $-M_3(x)=M_1(x)$.
	\item Let $a\in M_2(x)$ and  $y\in a+M_1(x)$. Then $a=x^{k_1}$ and $y=x^{k_1}+x^{k_2}=x^{k_1+k_2}$, where $k_1\equiv 2 \pmod 4$ and $k_2 \equiv 1 \pmod 4$. Since $k_1+k_2\equiv3 \pmod 4$, we have $y\in M_3(x)$ implying that $a+M_1(x)\subseteq M_3(x)$. Since size of both sets $M_1(x)$ and $M_3(x)$ are same, hence $a+M_1(x)=M_3(x)$. Similarly, $a+M_3(x)=M_1(x)$ for all  $a\in M_2(x)$.
	\item The proof is similar to Part (iv).
\end{enumerate} 	
\end{proof}

\begin{lema}\label{imcgoa11}
Let $x\in \Gamma(4)$. Then $i\bigg(\sum\limits_{s\in M_1(x)} \psi_{\alpha}(s)-\sum\limits_{s\in M_3(x)}\psi_{\alpha}(s)\bigg) \in \mathbb{Z}$ for all $\alpha \in \Gamma$.
\end{lema}
\begin{proof} Let $x\in \Gamma(4)$, $\alpha \in \Gamma$ and  $$\mu_{\alpha}=i\bigg(\sum\limits_{s\in M_1(x)} \psi_{\alpha}(s)-\sum\limits_{s\in M_3(x)}\psi_{\alpha}(s)\bigg).$$\\
Case 1: There exists $a\in M_2(x)$ such that $\psi_{\alpha}(a)\neq -1$. Then 
\begin{equation*}
\begin{split}
\mu_{\alpha}&=-i \bigg(\sum_{s\in M_3(x)} \psi_{\alpha}(s)-\sum_{s\in M_1(x)} \psi_{\alpha}(s)\bigg)\\
&=-i \bigg(\sum_{s\in a+M_1(x)} \psi_{\alpha}(s)-\sum_{s\in a+M_3(x)} \psi_{\alpha}(s)\bigg)\\
&=-i \bigg(\sum_{s\in M_1(x)} \psi_{\alpha}(a+s)-\sum_{s\in M_3(x)} \psi_{\alpha}(a+s)\bigg)\\
&=-i \psi_{\alpha}(a) \bigg(\sum_{s\in M_1(x)} \psi_{\alpha}(s)-\sum_{s\in M_3(x)} \psi_{\alpha}(s)\bigg)\\
&=- \psi_{\alpha}(a) \mu_{\alpha},
\end{split}
\end{equation*} We have $(1+\psi_{\alpha}(a)) \mu_{\alpha}=0$. Since $\psi_{\alpha}(a)\neq -1$, so $\mu_{\alpha}=0\in \mathbb{Z}$.\\
Case 2: There exists $a\in M_0(x)$ such that $\psi_{\alpha}(a)\neq 1$. Applying the same process as Case 1, we get $\mu_{\alpha}=0\in \mathbb{Z}$. \\
Case 3: Assume that $\psi_{\alpha}(a)=-1$ for all $a\in M_2(x)$ and $\psi_{\alpha}(a)=1$ for all $a\in M_0(x)$. Then \linebreak[4]$\psi_{\alpha}(a)=-\psi_{\alpha}(x)$ for all $a\in M_3(x)$ and $\psi_{\alpha}(a)=\psi_{\alpha}(x)$ for all $a\in M_1(x)$. Therefore,
\begin{equation*}
\begin{split}
\mu_{\alpha}&=i \bigg(\sum_{s\in M_1(x)} \psi_{\alpha}(s)-\sum_{s\in M_3(x)} \psi_{\alpha}(s)\bigg)\\
&= 2i\psi_{\alpha}(x) |M_1(x)|.
\end{split}
\end{equation*}
Since $\psi_{\alpha}(x)^4=1$ and $\mu_{\alpha}$ is a real number, we have $\psi_{\alpha}(x)=\pm i$. Thus $\mu_{\alpha} = \pm 2 |M_1(x)| \in \mathbb{Z}$.
\end{proof}

For $m\equiv 0 \pmod 4$ and $r\in \{1,3 \}$, define $$G_m^r(1)=\{ k: k\equiv r\Mod 4, \gcd(k,m )= 1 \}.$$

Define $\Gamma(4)$ to be the set of all $x\in \Gamma$ which satisfies $ord(x)\equiv 0 \pmod 4$. It is clear that  $exp(\Gamma) \equiv 0\pmod 4$ if and only if $\Gamma(4)\neq \emptyset$. Define an equivalence relation $\approx$ on $\Gamma(4)$ such that $x \approx y$ if and only if $y=x^k$ for some  $k\in G_m^1(1)$, where $m=ord(x)$. Observe that if  $x,y\in \Gamma(4)$ and $x \approx y$ then $x \sim y$, but the converse need not be true. For example, consider $x=5\pmod {12}$, $y=11\pmod {12}$ in $\mathbb{Z}_{12}$. Here $x,y\in \mathbb{Z}_{12}(4)$ and $x \sim y$ but $x \not\approx y$. For $x\in \Gamma(4)$, let $\llbracket x \rrbracket$ denote the equivalence class of $x$ with respect to the relation $\approx$.

\begin{lema}\label{lemanecc} Let $\Gamma$ be an abelian group, $x\in \Gamma(4)$ and $m=ord(x)$. Then the following are true.
\begin{enumerate}[label=(\roman*)]
\item $\llbracket x \rrbracket= \{ x^k:  k \in G_m^1(1)  \}$.
\item $\llbracket -x \rrbracket= \{ x^k: k \in G_m^3(1)  \}$.
\item $\llbracket x \rrbracket \cap \llbracket -x \rrbracket=\emptyset$.
\item $[x]=\llbracket x \rrbracket \cup \llbracket -x \rrbracket$.
\end{enumerate} 
\end{lema}
\begin{proof}
\begin{enumerate}[label=(\roman*)]
\item Let $y\in \llbracket x \rrbracket$. Then $x \approx y$, and so $ord(x)=ord(y)=m$ and there exists $k\in G_m^1(x)$ such that $y=x^k$. Thus $\llbracket x \rrbracket\subseteq \{ x^k:  k \in G_m^1(1)  \}$. On the other hand, let $z=x^k$ for some $k\in G_m^1(1)$. Then $ord(x)=ord(z)$ and so $x \approx z$. Thus $ \{ x^k:  k \in G_m^1(1)  \} \subseteq \llbracket x \rrbracket$.
\item Note that $-x=x^{m-1}$ and $m-1\equiv 3 \pmod 4$. By Part $(i)$, 
\begin{equation*}
\begin{split}
\llbracket -x \rrbracket&= \{ (-x)^k:  k \in G_m^1(1)  \}\\
&= \{ x^{(m-1)k}:  k \in G_m^1(1)  \}\\
&= \{ x^{-k}:  k \in G_m^1(1)  \}\\
&= \{ x^{k}:  k \in G_m^3(1)  \}.
\end{split}
\end{equation*}
\item Since $G_m^1(1)\cap G_m^3(1)=\emptyset$, so by Part $(i)$ and Part $(ii)$, $\llbracket x \rrbracket \cap \llbracket -x \rrbracket=\emptyset$ holds.
\item  Since $[x]= \{ x^k:  k \in G_m(1)  \}$ and $G_m(1)$ is a disjoint union of $G_m^1(1)$ and $ G_m^3(1)$, by Part $(i)$ and Part $(ii)$, $[x]=\llbracket x \rrbracket \cup \llbracket -x \rrbracket$ holds.
\end{enumerate} 
\end{proof}

Let $D_g$ be the set of all odd divisors of $g$, and $D_g^1$ (resp. $D_g^3$) be the set of all odd divisors of $g$ which are congruent to $1$ (resp. $3$) modulo $4$. It is clear that $D_g=D_g^1 \cup D_g^3$.

\begin{lema}\label{imcgoa4} Let $\Gamma$ be an abelian group, $x\in \Gamma(4)$, $m=ord(x)$ and $g=\frac{m}{4}$. Then the following are true.
\begin{enumerate}[label=(\roman*)]
\item $M_1(x) \cup M_3(x)=\bigcup\limits_{h\in D_g} [x^h] $.
\item $M_1(x)= \bigcup\limits_{h\in D_g^1} \llbracket x^h \rrbracket \cup \bigcup\limits_{h\in D_g^3} \llbracket -x^h \rrbracket$.
\item $M_3(x)=\bigcup\limits_{h\in D_g^1} \llbracket -x^h \rrbracket \cup \bigcup\limits_{h\in D_g^3} \llbracket x^h \rrbracket $.
\end{enumerate}
\end{lema}
\begin{proof}
\begin{enumerate}[label=(\roman*)]
\item Let $x^k \in M_1(x) \cup M_3(x)$, where $k \equiv 1 \text{ or } 3 \pmod 4$. To show that $x^k \in \bigcup\limits_{h\in D_g} [x^h]$, it is enough to show $x^k \sim x^h$ for some $h \in D_g$. Let $h=\gcd (k,g) \in D_g$. Note that 
$$ord(x^k)=\frac{m}{\gcd(m,k)}=\frac{m}{\gcd(g,k)}=\frac{m}{h}=ord(x^h).$$ 
Also, as $h=\gcd(k,m)$, we have $\langle x^k \rangle = \langle x^h \rangle$, and so $x^k=x^{hj}$ for some $j\in G_q(1)$, where $q=ord(x^h)=\frac{m}{h}$. Thus $x^k\sim x^h$ where $h=\gcd (k,g) \in D_g$. 
Conversely, let $z\in \bigcup\limits_{h\in D_g} [x^h]$. Then there exists $h\in D_g$ such that $z=x^{hj}$ where $j\in G_q(1)$ and $q= \frac{m}{\gcd(m,h)}$. Now $h\in D_g$ and $q\equiv 0\pmod 4$ imply that both $h$ and $j$ are odd integers. Thus $hj\equiv 1 \text{ or } 3 \pmod 4$ and so $\bigcup\limits_{h\in D_g} [x^h] \subseteq M_1(x) \cup M_3(x)$. Hence $M_1(x) \cup M_3(x)=\bigcup\limits_{h\in D_g} [x^h] $.
	
\item Let $x^k \in M_1(x)$, where $k\equiv 1 \pmod 4$. By Part $(i)$, there exists $h\in D_g$  and $j\in G_q(1)$ such that $x^k=x^{hj}$, where $q=\frac{m}{\gcd(m,h)}$. Note that $k=jh$. If $h\equiv 1 \pmod 4$ then $j\in G_q^1(1)$, otherwise $j\in G_q^3(1)$. Thus using parts $(i)$ and $(ii)$ of Lemma \ref{lemanecc}, if $h\equiv 1 \pmod 4$ then $x^k \approx x^h$, otherwise $x^k \approx -x^h$. Hence $M_1(x) \subseteq \bigcup\limits_{h\in D_g^1} \llbracket x^h \rrbracket \cup \bigcup\limits_{h\in D_g^3} \llbracket -x^h \rrbracket$. Conversely, assume that $z\in  \bigcup\limits_{h\in D_g^1} \llbracket x^h \rrbracket \cup \bigcup\limits_{h\in D_g^3} \llbracket -x^h \rrbracket$. This gives $z\in \llbracket x^h \rrbracket$ for an $h\in D_g^1$ or $z\in \llbracket -x^h \rrbracket$ for an $h\in D_g^3$.  In the first case, by part $(i)$ of Lemma \ref{lemanecc},  there exists  $j\in G_q^1(1)$ with $q=\frac{m}{\gcd(m,h)}$ such that  $z = x^{hj}$. Similarly, for the second case, by part $(ii)$ of Lemma \ref{lemanecc},  there exists  $j\in G_q^3(1)$ with $q=\frac{m}{\gcd(m,h)}$ such that  $z = x^{hj}$. In both the cases, $hj \equiv 1 \pmod 4$. Thus $z\in M_1(x)$.

\item The proof is similar to Part $(ii)$.
\end{enumerate}
\end{proof}

\begin{lema}\label{integerEigenvalue}
Let $\Gamma$ be an abelian group and $x\in \Gamma(4)$. Then $i\bigg( \sum\limits_{s\in \llbracket x \rrbracket} \psi_{\alpha}(s)-\sum\limits_{s\in \llbracket -x \rrbracket}\psi_{\alpha}(s)\bigg) \in \mathbb{Z}$ for all $\alpha \in \Gamma$.
\end{lema}
\begin{proof} Note that there exists $x\in \Gamma(4)$ with $ord(x)=4$. Apply induction on $ord(x)$. If $ord(x)=4$, then $M_1(x)=\llbracket x \rrbracket$ and $M_3(x)=\llbracket -x \rrbracket$. Hence by Lemma ~\ref{imcgoa11}, $i\bigg( \sum\limits_{s\in \llbracket x \rrbracket} \psi_{\alpha}(s)-\sum\limits_{s\in \llbracket -x \rrbracket}\psi_{\alpha}(s)\bigg) \in \mathbb{Z}$ for all $\alpha \in \Gamma$. Assume that the statement  holds for all $x\in \Gamma(4)$ with $ord(x)\in \{ 4,8,...,4(g-1)\}$. We prove it for $ord(x)=4g$. Lemma ~\ref{imcgoa4} implies that
$$M_1(x)= \bigcup\limits_{h\in D_g^1} \llbracket x^h \rrbracket \cup \bigcup\limits_{h\in D_g^3} \llbracket -x^h \rrbracket$$ and $$M_3(x)=\bigcup\limits_{h\in D_g^1} \llbracket -x^h \rrbracket \cup \bigcup\limits_{h\in D_g^3} \llbracket x^h \rrbracket .$$ If $ord(x)=4g=m$ and $h>1$ then $ord(x^h), ord(-x^h)\in \{ 4,8,...,4(g-1)\}$. By induction hypothesis 
$$i\bigg( \sum\limits_{s\in \llbracket x^h \rrbracket} \psi_{\alpha}(s)-\sum\limits_{s\in \llbracket -x^h \rrbracket}\psi_{\alpha}(s)\bigg) \in \mathbb{Z} \textnormal{ for all }\alpha \in \Gamma.$$
Now we have 
\begin{equation*}
\begin{split}
i\bigg(\sum\limits_{s\in M_1(x)} \psi_{\alpha}(s)-\sum\limits_{s\in M_3(x)}\psi_{\alpha}(s)\bigg) &= i\bigg( \sum\limits_{s\in \llbracket x \rrbracket} \psi_{\alpha}(s)-\sum\limits_{s\in \llbracket -x \rrbracket}\psi_{\alpha}(s)\bigg) \\
&+ \sum_{h\in D_{g}^1, h> 1} i\bigg( \sum\limits_{s\in \llbracket x^h \rrbracket} \psi_{\alpha}(s)-\sum\limits_{s\in \llbracket -x^h \rrbracket}\psi_{\alpha}(s)\bigg)\\
&+ \sum_{h\in D_{g}^3, h> 1} i\bigg( \sum\limits_{s\in \llbracket -x^h \rrbracket} \psi_{\alpha}(s)-\sum\limits_{s\in \llbracket x^h \rrbracket}\psi_{\alpha}(s)\bigg).
\end{split} 
\end{equation*}
Hence
\begin{equation*}
\begin{split}
i\bigg( \sum\limits_{s\in \llbracket x \rrbracket} \psi_{\alpha}(s)-\sum\limits_{s\in \llbracket -x \rrbracket}\psi_{\alpha}(s)\bigg) &= i\bigg(\sum\limits_{s\in M_1(x)} \psi_{\alpha}(s)-\sum\limits_{s\in M_3(x)}\psi_{\alpha}(s)\bigg)\\
&- \sum_{h\in D_{g}^1, h> 1} i\bigg( \sum\limits_{s\in \llbracket x^h \rrbracket} \psi_{\alpha}(s)-\sum\limits_{s\in \llbracket -x^h \rrbracket}\psi_{\alpha}(s)\bigg)\\
&+ \sum_{h\in D_{g}^3, h> 1} i\bigg( \sum\limits_{s\in \llbracket x^h \rrbracket} \psi_{\alpha}(s)-\sum\limits_{s\in \llbracket -x^h \rrbracket}\psi_{\alpha}(s)\bigg)
\end{split} 
\end{equation*} is also an integer for all $\alpha \in \Gamma$ because of Lemma ~\ref{imcgoa11} and induction hypothesis.
\end{proof}

 For $exp(\Gamma) \equiv 0 \pmod 4$, define $\mathbb{D}(\Gamma)$ to be the set of all skew symmetric subsets $S$ of $\Gamma$ such that $S=\llbracket x_1 \rrbracket\cup...\cup \llbracket x_k \rrbracket$ for some $x_1,...,x_k\in \Gamma(4)$. For  $exp(\Gamma) \not\equiv 0 \pmod 4$, define $\mathbb{D}(\Gamma)=\{ \emptyset \}$.
 
\begin{theorem}\label{OrientedChara}
Let $\Gamma$ be an abelian group. If $S \in \mathbb{D}(\Gamma)$ then the oriented Cayley graph  $Cay(\Gamma, S)$ is integral.
\end{theorem}
\begin{proof}
%In view of Theorem~\ref{neccori}, we only need to prove the converse part. 
Assume that $S \in \mathbb{D}(\Gamma)$. Then $S=\llbracket x_1 \rrbracket\cup...\cup \llbracket x_k \rrbracket$ for some $x_1,...,x_k\in \Gamma(4)$. Let $Sp_H(Cay(\Gamma, S))=\{ \mu_\alpha : \alpha \in \Gamma \}$. We have 	
\begin{equation*}
	\begin{split}
\mu_{\alpha} &=i\sum_{s\in S}\bigg( \psi_{\alpha}(s)-\psi_{\alpha}(-s)\bigg) \\
&= \sum_{j=1}^k \sum_{s\in \llbracket x_j \rrbracket} i\bigg( \psi_{\alpha}(s)-\psi_{\alpha}(-s)\bigg).
	\end{split} 
\end{equation*}
Now by Lemma~\ref{integerEigenvalue}, $\mu_{\alpha} \in \mathbb{Z}$ for all $\alpha \in \Gamma$. Hence the oriented Cayley graph  $Cay(\Gamma, S)$ is integral.
\end{proof}

\begin{theorem}
Let $\Gamma$ be an abelian group. If $S \setminus \overline{S} \in \mathbb{B}(\Gamma)$ and $\overline{S} \in \mathbb{D}(\Gamma)$ then the mixed Cayley graph  $Cay(\Gamma, S)$ is integral.
\end{theorem}
\begin{proof} By Theorem~\ref{smallchara}, $Cay(\Gamma, S)$ is integral if and only if both $Cay(\Gamma, S \setminus \overline{S})$ and $Cay(\Gamma, \overline{S})$ are integral. Thus the result follows from Theorem~\ref{Cayint} and Theorem~\ref{OrientedChara}.
\end{proof}

%%%%%%%%%%%%%%%%%%%%%%%%%%%%%%%%%%%%%%%%%%%%%%%%%%%%%%%%%%%%%

\section{Characterization of integral mixed Cayley graph over abelian group}

	The \textit{cyclotomic polynomial} $\Phi_n(x)$ is the monic polynomial whose zeros are the primitive $n^{th}$ root of unity. That is $$\Phi_n(x)= \prod_{a\in G_n(1)}(x-w_n^a),$$ where $w_n=\exp(\frac{2\pi i}{n})$. Clearly the degree of $\Phi_n(x)$ is $\varphi(n)$. See \cite{numbertheory} for more details about cyclotomic polynomials.

\begin{theorem}~\cite{numbertheory}
	The cyclotomic polynomial $\Phi_n(x)$ is irreducible in $\mathbb{Z}[x]$.
\end{theorem}

The polynomial $\Phi_n(x)$ is irreducible over $\mathbb{Q}(i)$ if and only if $[\mathbb{Q}(i,w_n) : \mathbb{Q}(i)]= \varphi(n)$. Also $ \mathbb{Q}(w_n)$ does not contain the number $i=\sqrt{-1}$ if and only if $n\not\equiv 0 \Mod 4$. Thus, if $n\not\equiv 0 \Mod 4$ then $[\mathbb{Q}(i,w_n):\mathbb{Q}(w_n) ]=2=[\mathbb{Q}(i), \mathbb{Q}]$, and therefore $$[\mathbb{Q}(i,w_n) : \mathbb{Q}(i)]=\frac{[\mathbb{Q}(i,w_n) : \mathbb{Q}(w_n)] . [\mathbb{Q}(w_n) : \mathbb{Q}]}{ [\mathbb{Q}(i) : \mathbb{Q}]}= [\mathbb{Q}(w_n) : \mathbb{Q}]= \varphi(n).$$ Hence for $n\not\equiv 0 \Mod 4$, the polynomial $\Phi_n(x)$ is irreducible over $\mathbb{Q}(i)$. 

Let $n \equiv 0 \Mod 4$. Then $ \mathbb{Q}(i,w_n)= \mathbb{Q}(w_n)$, and so $$[\mathbb{Q}(i,w_n) : \mathbb{Q}(i)] = \frac{[\mathbb{Q}(i,w_n) : \mathbb{Q}]}{[\mathbb{Q}(i) : \mathbb{Q}]}=\frac{\varphi(n)}{2}.$$ Hence the polynomial $\Phi_n(x)$ is reducible over $\mathbb{Q}(i)$.

We know that $G_n(1)$ is a disjoint union of $G_n^1(1)$ and $G_n^3(1)$. Define $$\Phi_n^{1}(x)= \prod_{a\in G_n^1(1)}(x-w_n^a) \textnormal{ and } \Phi_n^3(x)= \prod_{a\in G_n^3(1)}(x-w_n^a).$$ It is clear from the definition that $\Phi_n(x)=\Phi_n^1(x)\Phi_n^3(x)$.

\begin{theorem}~\cite{kadyan2021integral}
	Let $n\equiv 0\Mod 4$. The factors $\Phi_n^1(x)$ and $\Phi_n^3(x)$ of $\Phi_n(x)$ are irreducible monic polynomials in $\mathbb{Q}(i)[x]$ of degree $\frac{\varphi(n)}{2}$.
\end{theorem}

In this section, first we prove that there is no integral oriented Cayley graph $Cay(\Gamma, S)$ for $exp(\Gamma) \not\equiv 0\pmod 4$ and $S\neq \emptyset$. After that we find a necessary condition on the set $S$ so that the mixed Cayley graph $Cay(\Gamma, S)$ is integral.

\begin{theorem}\label{ori4}
Let $\Gamma$ be an abelian group and $exp(\Gamma) \not\equiv 0\pmod 4$. Then the oriented Cayley graph $Cay(\Gamma, S)$ is integral if and only if $S=\emptyset$
\end{theorem}
\begin{proof}  
Let $l=exp(\Gamma)$ and $Sp_H(Cay(\Gamma, S))=\{ \mu_\alpha : \alpha \in \Gamma \}$. Assume that $l\not\equiv 0\pmod 4$ and $Cay(\Gamma, S)$ is integral. By Corollary~\ref{OriEig}, $\mu_{\alpha}=-\mu_{-\alpha}\in \mathbb{Q}$  and $$\mu_{\alpha}=i\sum_{s\in S}\bigg( \psi_{\alpha}(s)-\psi_{\alpha}(-s)\bigg) \textnormal{ for all } \alpha \in \Gamma.$$ 
Note that, $\psi_{\alpha}(s)$ and $\psi_{\alpha}(-s)$ are $l^{th}$ roots of unity for all $ \alpha \in \Gamma, s\in S$.
Fix a primitive $l^{th}$ root $w$ of unity and express $\psi_{\alpha}(s)$ in the form $w^j$ for some $j \in \{ 0,1,...,l-1\}$. Thus $$\mu_{\alpha}=i\sum_{s\in S}\bigg( \psi_{\alpha}(s)-\psi_{\alpha}(-s)\bigg) = \sum_{j=0}^{l-1} a_j w^j,$$
 where $a_j \in \mathbb{Q}(i)$. Since $\mu_{\alpha} \in \mathbb{Q}$, so $p(x)= \sum\limits_{j=0}^{l-1} a_j x^j- \mu_{\alpha} \in \mathbb{Q}(i)[x]$ and $w$ is a root of $p(x)$. Since $l\not\equiv 0(\mod 4)$, so $\Phi_l(x)$ is irreducible in $\mathbb{Q}(i)[x]$. Therefore, $p(x)$ is a multiple of the irreducible polynomial $\Phi_l(x)$, and so $w^{-1}=w^{n-1}$ is also a root of $p(x)$. Note that, if $\psi_{\alpha}(s)=w^j$ for some $j \in \{ 0,1,...,l-1\}$ then $\psi_{-\alpha}(s)=w^{-j}$. We have 
 $$0=p(w^{-1})=\sum_{j=0}^{l-1} a_j w^{-j}- \mu_{\alpha}=\mu_{-\alpha}-\mu_{\alpha} \Rightarrow \mu_\alpha = \mu_{-\alpha}.$$ 
Since $\mu_{-\alpha}=-\mu_{\alpha}$, we get $\mu_{\alpha}=0$, for all $\alpha \in \Gamma$. Hence $S=\emptyset$. 

Conversely, if $S=\emptyset$ then all the eigenvalues of $Cay(\Gamma, S)$ are zero. Thus $Cay(\Gamma, S)$ is integral.
\end{proof}

Lemma \ref{lemanecc} says  that corresponding to each equivalence class of the relation $\sim$ we get two equivalence classes of the relation $\approx$.
 Define $E$ to be the matrix of size $n\times n$,  whose rows and columns are indexed by elements of $\Gamma$ such that $E_{x,y}=i\psi_{x}(y)$. Note that each row of $E$ corresponds to a character of $\Gamma$ and $EE^*=nI_n$. Let $v_{\llbracket x \rrbracket}$ be the vector in $\mathbb{Q}^n$ whose coordinates are indexed by elements of $\Gamma$ for which the $z^{th}$ coorninate is $1$ and $(n-z)^{th}$ coorninate is $-1$ for all $z \in \llbracket x \rrbracket$. By Lemma~\ref{integerEigenvalue}, $Ev_{\llbracket x \rrbracket} \in \mathbb{Q}^n$. 

\begin{lema}\label{Ori4Nec} Let $\Gamma$ be an abelian group, $v\in \mathbb{Q}^n$ and $Ev \in \mathbb{Q}^n$. Let the coordinates of $v$ be indexed by elements of $\Gamma$. Then
\begin{enumerate}[label=(\roman*)]
\item $v_x=-v_{-x}$ for all $x \in \Gamma$.
\item $v_x=v_y$ for all $x,y \in \Gamma(4)$ satisfying $x \approx y$.
\item $v_x=0$ for all $x\in \Gamma \setminus \Gamma(4)$.
\end{enumerate} 
\end{lema}
\begin{proof} Let $E_x$ and $E_y$ denote the column vectors of $E$ indexed by $x$ and $y$, respectively, and assume that $u=Ev \in \mathbb{Q}^n$.
\begin{enumerate}[label=(\roman*)]
\item Since $v=\frac{1}{n}E^*u$, we have
 $$v_x = \frac{1}{n}(E^* u)_x \Rightarrow \overline{v}_x= \frac{1}{n} \overline{(E^* u_x)} = - \frac{1}{n} (E^* u)_{-x} = -v_{-x}.$$
 Since $v_x=\overline{v}_x$, it gives $v_x=-v_{-x}$ for all $x \in \Gamma$.
\item If $\Gamma(4) = \emptyset$ then there is nothing to prove. Now assume that $\Gamma(4)\neq\emptyset$, so that $exp(\Gamma)\equiv 0 \pmod 4$.
Let $x,y \in \Gamma(4)$ and $x \approx y$. Then there exists $k\in G_m^1(1)$ such that $y=x^k$, where $m=ord(x)=ord(y)$. Assume that $x\neq y$, so that $k\geq 2$.
Using Lemma ~\ref{Basic}, entries of $E_x$ and $E_y$ are $i$ times an $m^{th}$ root of unity. Fix a primitive $m^{th}$ root of unity $w$, and express each entry of $E_x$ and $E_y$ in the form $iw^j$ for some $j\in \{ 0,1,...,m-1\}$. 
Thus $$nv_x= (E^*u)_x= \sum_{j=0}^{m-1} a_j w^j,$$ where $a_j\in \mathbb{Q}(i)$ for all $j$. Thus $w$ is a root of the polynomial $p(x)= \sum\limits_{j=0}^{m-1} a_j x^j-nv_x \in \mathbb{Q}(i)[x]$. Therefore, $p(x)$ is a multiple of the irreducible polynomial $\Phi_m^1(x)$, and so $w^k$ is also a root of $p(x)$, because of $k\in G_m^1(1)$.
As $y=x^k$ implies that $\psi_{y}(a)=\psi_{x}(a)^k$ for all $a\in \Gamma$, we have $(E^*u)_y= \sum\limits_{j=0}^{m-1} a_j w^{kj}$. Hence 
$$0 =p(w^k)=  \sum\limits_{j=0}^{m-1} a_j w^{kj}-nv_x= (E^*u)_y -nv_x=nv_y-nv_x \Rightarrow v_x=v_y.$$

\item Let $x\in \Gamma \setminus \Gamma(4)$ and $r=ord(x) \not\equiv 0 \pmod 4$. Fix a primitive $r^{th}$ root $w$ of unity, and express each entry of $E_x$ in the form $iw^j$ for some $j\in \{ 0,1,...,r-1\}$. 
Thus $$nv_x= (E^*u)_x= \sum_{j=0}^{r-1} a_j w^j,$$ where $a_j\in \mathbb{Q}(i)$ for all $j$. Thus $w$ is a root of the polynomial $p(x)= \sum\limits_{j=0}^{r-1} a_j x^j-nv_x \in \mathbb{Q}(i)[x]$. Therefore, $p(x)$ is a multiple of the irreducible polynomial $\Phi_r(x)$, and so $w^{-1}$ is also a root of $p(x)$. Since $\psi_{-x}(a)=\psi_{x}(a)^{-1}$ for all $a\in \Gamma$, therefore $(E^*u)_{-x}= \sum\limits_{j=0}^{r-1} a_j w^{-j}$. Hence $$0 =p(w^{-1})=  \sum\limits_{j=0}^{r-1} a_j w^{-j}-nv_x= (E^*u)_{-x} -nv_x=nv_{-x}-nv_x,$$ implies that $v_x=v_{-x}$. This together with Part $(i)$ imply that $v_x=0$ for all $x\in \Gamma \setminus \Gamma(4)$.
\end{enumerate} 
\end{proof}

\begin{theorem}\label{neccori}
Let $\Gamma$ be an abelian group. The oriented Cayley graph  $Cay(\Gamma, S)$ is integral if and only if $S \in \mathbb{D}(\Gamma)$.
\end{theorem}
\begin{proof}
If $exp(\Gamma) \not\equiv 0 \pmod 4$ then by Theorem ~\ref{ori4}, we have $S = \emptyset$, and so $S \in \mathbb{D}(\Gamma)$. Now assume that $exp(\Gamma) \equiv 0 \pmod 4$. Let $e_S$ be the vector in $\mathbb{R}^n$ whose $j^{th}$ coordinate is $1$ and $(n-j)^{th}$ coordinate is $-1$ for all $j\in S$. As $Cay(\Gamma, S)$ is integral implies that $Ee_S \in \mathbb{Q}^n$, so by Lemma ~\ref{Ori4Nec}, we get $S \in \mathbb{D}(\Gamma)$, as desired. Converse follows from Theorem~\ref{OrientedChara}.
\end{proof}

\begin{theorem}
Let $\Gamma$ be an abelian group. The mixed Cayley graph  $Cay(\Gamma, S)$ is integral if and only if $S \setminus \overline{S} \in \mathbb{B}(\Gamma)$ and $\overline{S} \in \mathbb{D}(\Gamma)$.
\end{theorem}
\begin{proof}
By Theorem~\ref{smallchara}, $Cay(\Gamma, S)$ is integral if and only if both $Cay(\Gamma, S \setminus \overline{S})$ and $Cay(\Gamma, \overline{S})$ are integral. Thus the result follows from Theorem~\ref{Cayint} and Theorem ~\ref{neccori}.
\end{proof}

%%%%%%%%%%%%%%%%%%%%%%%%%%%%%%%%%%%%%%%%%%%%%%%%%%%%%%%%%%%%%%%%%

\end{document}